\documentclass[11pt, reqno]{amsart}
\usepackage{graphicx}
\usepackage{amssymb}
\usepackage{array}
\usepackage{enumerate}
\usepackage{url}
\usepackage{algorithm}
\usepackage{dsfont}
\usepackage{diagbox}
\usepackage{mathtools}
\usepackage{float}
\usepackage{wrapfig}
\usepackage{lscape}
\usepackage{rotating}
\usepackage{epstopdf}
\usepackage{mathtools}

\newtheorem{thm}{Theorem}[section]

\theoremstyle{plain}

\newtheorem{corollary}[thm]{Corollary}

\newtheorem{lemma}[thm]{Lemma}

\newtheorem{theorem}[thm]{Theorem}

\newtheorem*{notation*}{Notation}

\newtheorem{remark}[thm]{Remark}

\newcommand\be{\begin{equation}}
\newcommand\ee{\end{equation}}
\newcommand\bea{\begin{eqnarray}}
\newcommand\eea{\end{eqnarray}}
\newcommand\bi{\begin{itemize}}
	\newcommand\ei{\end{itemize}}
\newcommand\ben{\begin{enumerate}[(a)]}
	\newcommand\een{\end{enumerate}}
\newcommand\bc{\begin{center}}
	\newcommand\ec{\end{center}}
\def\ba#1\ea{\begin{align*}#1\end{align*}}

\usepackage{enumerate}

\allowdisplaybreaks
\begin{document}
\setlength{\abovedisplayskip}{1pt}
\setlength{\belowdisplayskip}{1pt}
\setlength{\abovedisplayshortskip}{1pt}
\setlength{\belowdisplayshortskip}{1pt}
\title{An Explicit Upper Bound for $|\zeta(1+it)|$}
	\author{Dhir Patel}
	\maketitle\vspace{-5ex}
\begin{abstract}
   In this paper we provide an explicit bound for $|\zeta(1+it)|$ in the form of $|\zeta(1+it)|\leq \min\left(\log t, \frac{1}{2}\log t+1.93, \frac{1}{5}\log t+44.02 \right)$. This improves on the current best-known explicit bound of $|\zeta(1+it)|\leq 62.6(\log t)^{2/3}$ up until $t$ of the magnitude $10^{10^7}$.
\end{abstract}

\section{introduction}
The study of the growth rate of $\zeta(1+it)$ has been of great interest because of its application in estimating $S(T)$ as shown in \cite{trudgianst} and computing zero free regions for the Riemann zeta function. 

In 1900 Mellin \cite{mellin} was the first to obtain a result in this direction and showed that for real $t$ bounded away from $0$ that
\begin{equation}
    \zeta(1+it)=\mathcal{O} (\log |t|) \label{mellin}
\end{equation}
In 1921, Weyl improved (\ref{mellin}) using Weyl's sums in \cite{weyl} to 
\begin{equation}
    \zeta(1+it)=\mathcal{O}\left(\frac{\log t}{\log \log t}\right)
\end{equation}
and was in turn improved upon by Vinogradov which can be found  in \textup{\cite[Theorem 6.14]{titchmarsh}} to
\begin{equation}
    \zeta(1+it)=\mathcal{O}(\log ^{3/4}t \log^{3/4} \log t) \label{vino}
\end{equation}

Several authours namely Flett \cite{flett}, Walfisz \cite{walfisz}, and Korobv \cite{korobov} between 1950-58 obtained bounds of the form $\mathcal{O}(\log^{3/4} t \log \log^{1/2+\epsilon} t),\mathcal{O}(\log^{3/4} t \log \log^{1/2} t), \mathcal{O}(\log^{5/7+\epsilon})$ respectively.
Moreover, authors such as Vinogradov \cite{vinogradov}, Korobov \cite{korobov1} \cite{korobov2} in 1958 and Richert \cite{richert} in 1967 gave the best known unconditional estimate 
\begin{equation}
    \zeta(1+it) =\mathcal{O}(\log^{2/3}t) \label{richert}
\end{equation}

There are several conditional bounds known for $\zeta(1+it)$. One such was given by Littlewood in 1912, assuming the Lindel\"{o}f hypothesis and showed that
\begin{equation}
    \zeta(1+it)=\mathcal{O}(\log \log t \log\log\log t)
\end{equation}
He further improved upon this result in 1928 \cite{littlewood} and provided the best known conditional bound assuming the Riemann Hypothesis that states
\begin{equation}
    \zeta(1+it)=\mathcal{O}(\log\log t).
\end{equation}

In addition to the asymptotic behaviour of $\zeta(1+it)$, many explicit bounds of the form 
\begin{equation}
    |\zeta(1+it)|\leq a\log t, \hspace{2mm} \text{for}\hspace{2mm} t\geq t_0 \label{landaugen}
\end{equation}
have also been derived for it. One of the earliest known results is given by
Landau \cite{landau} in 1903 where he shows $a=2$, $t_0=10$. Backlund \cite{backlund} in 1918 improved this result to $a=1$ and $t_0>50$ and this $t_0$ was lowered by Trudgian in \cite{trudgian} to $2.001\ldots$ and in the same paper he showed\footnote[1]{This explicit estimate with $a = 3/4$ is obtained as an application of an explicit van der Corput test using the second derivative derived by Cheng-Graham in \textup{\cite[Lemma 3]{graham}}. However, that Cheng-Graham result is now known to be incorrect. This seems to be an irrecoverable error and the $|\zeta(1+it)|$ estimate in \cite{trudgian} no longer holds. This is elaborated further in section 2.} that $a=\frac{3}{4}$ and $t_0=3$.

The best known explicit bounds for large $t$ are of the form
\begin{equation}
    |\zeta(1+it)|\leq A\log^{2/3}t, \hspace{2mm} \text{for}\hspace{2mm} t\geq t_0 \label{expbest}
\end{equation}
In 1967, Richert first obtained $(\ref{expbest})$ in \cite{richert} for unknown constant $A$ which was computed in 1985 by Ellison\cite{ellison} to be $2100$ with $t_0=3$ and in 1995 Cheng \cite{cheng} improved it to $175$ with $t_0=2$. More recent improvements have been given by Ford in  \cite{ford} where he showed $A=72.6$ which Trudgian improved to $A=62.6$ with $t_0=3$ in \cite{trudgian}.

Moreover, bounds such as $(\ref{expbest})$ seem to improve over bounds like $(\ref{landaugen})$ only when $t$ is astronomically large because of extremely large $A$ value. Hence, it is worth obtaining good explict bounds of the form $(\ref{landaugen})$ for computational purposes when $t$ is relatively small. Keeping this in mind we have the following theorem 
\begin{theorem}\label{maintheorem}
If $t\geq 3$, then 
\begin{equation}
    \lvert \zeta(1+it)\rvert\leq \min\left(\log t, \frac{1}{2}\log t+1.93, \frac{1}{5}\log t+44.02 \right) \label{actualzeta1const}
\end{equation}
In particular for $t\geq 8.261\ldots \times 10^{60}$
\begin{equation}
    \lvert \zeta(1+it)\rvert\leq \frac{1}{5}\log t+ 44.02 \label{eventualconst}
\end{equation}
\end{theorem} 

\section{Remark on Erroneous computational lemma in literature}
An important tool used to obtain bounds such as $(\ref{landaugen})$ with $a=\frac{3}{4}$ is an explicit version of van der Corput's second derivative test. This result can be found in the work of Cheng-Graham in \textup{\cite[Lemma 3]{graham}}. However a computational flaw was discovered by Kevin Ford \cite{xyz} and Reyna in \cite{reynacorrection} with this Cheng-Graham lemma. This affects many explicit estimates in the literature such as \cite{ford}, \cite{trudgian}, \cite{hiary} to name a few. However, we note that even if $(\ref{landaugen})$ with $a=\frac{3}{4}$ were true, Theorem $\ref{maintheorem}$ gives an improvement on it for $t\geq 2.17\times 10^3$.

For the remainder of the section we record the errors in the work of Cheng-Graham \textup{\cite[Lemma 2 and 3]{graham}}. We also correct another result in literature \textup{\cite[Lemma 1.2]{hiary}} affected due to these errors. 
We begin by providing corrected version of flawed Cheng-Graham lemma in \textup{\cite[Lemma 2]{graham}}. To do so, we first define 
$\|x\| \coloneqq \min_{n\in \mathbb{Z}} |x-n|$ and we observe that $0\leq \|x\|\leq 1/2.$

\begin{lemma}\label{incorrectevery}
Suppose $f$ is a continuously differentiable real-valued function with a monotonic derivative and $\|f'\|\geq U^{-1}$ for some positive real number $U$ on the interval $(a, b]$. Then
\begin{equation}
    |S|=\left\lvert\sum_{n\in(a, b]}e^{2\pi i f(n)}\right\rvert\leq \frac{2}{\pi}U. \label{correct}
\end{equation}
\end{lemma}
\begin{remark}
This result is often attributed to the works of Kuzmin-Landau \cite{kuzminver, landauineq} in literature.
\end{remark}
\begin{proof}
First we notice that $0<U^{-1}\leq 1/2$. Next, we make note of the error in the proof of Lemma 2 in \cite{graham} and give a possible fix. This fix is based on the ideas found in \cite{mordell}. Most of the proof of Lemma 2 in \cite{graham} is valid except we note a typo on Page 1266 where the equality for $G(n)-G(n-1)$ should read
\begin{equation}
    G(n)-G(n-1) = \frac{1}{2i}(\cot(\pi g(n-1))-\cot(\pi g(n)))
\end{equation}

However, there is a fatal flaw that originates in the first inequality on Page 1267 of \cite{graham} because of missing absolute values after the first two cotangent terms. That inequality should instead read:
\begin{align}
    \left\lvert\sum_{n=L}^{M}e^{2\pi i f(n)}\right\rvert&\leq \frac{1}{2}(\cot(\pi g(L))-\cot(\pi g(M-1)))+\left\lvert \frac{1}{2}+\frac{i}{2}\cot(\pi g(L))\right\rvert+\left\lvert\frac{1}{2}-\frac{i}{2}\cot(\pi g(M-1))\right\rvert \label{ineq0}
\end{align}
With this fix in mind, we now provide a possible way to finish Cheng-Graham proof correctly giving us $(\ref{correct})$. To do so, we note:
\begin{align}
     \left\lvert\sum_{n=L}^{M}e^{2\pi i f(n)}\right\rvert&\leq \frac{1}{2}\left(\cot(\pi g(L))+\frac{1}{\sin(\pi g(L))}\right) +\frac{1}{2}\left(\frac{1}{\sin(\pi g(M-1))}-\cot (\pi g(M-1))\right)\label{ineq1}\\
    &=\frac{\cos (\pi g(L))+1}{2\sin (\pi g(L))}+\frac{1-\cos(\pi g(M-1))}{2\sin(\pi g(M-1)) \label{ineq2}}\\
    &=\frac{1}{2}\cot\left(\frac{\pi g(L)}{2}\right)+\frac{1}{2}\tan \left(\frac{\pi g(M-1)}{2}\right)\label{ineq3}\\
    &\leq \frac{1}{2}\cot\left(\frac{\pi U^{-1}}{2}\right)+\frac{1}{2}\tan \left(\frac{\pi (1-U^{-1})}{2}\right)\label{ineq4}\\
    &\leq \cot \left(\frac{\pi U^{-1}}{2}\right)\label{ineq5}\\
    &\leq \frac{2}{\pi}U \label{ineq6}
\end{align} 

For proof readability we make several remarks regarding the above inequalities here. First note that going from $(\ref{ineq0})$ to $(\ref{ineq1})$ we use
\begin{align}
 0<U^{-1}\leq &g(L), \hspace{2mm} g(M-1)\leq 1-U^{-1}<1, \label{relation}\\
     1+i\cot(x)&=\frac{ie^{-ix}}{\sin x},  \hspace{3mm} \sin(\pi x)>0 \hspace{2mm}\text{for}\hspace{2mm} x\in(0, 1). \label{cotcos}
\end{align}
Next to go from $(\ref{ineq1})$ to $(\ref{ineq2})$ we write $\cot(x)$ in terms of $\sin(x)$ and $\cos(x)$ and then gather like terms. To pass from $(\ref{ineq2})$ to $(\ref{ineq3})$ we use the following relation valid for $x\neq k\pi$ where $k\in \mathbb{Z}$:
\begin{equation*}
    \cot \left(\frac{x}{2}\right)= \frac{1+\cos x}{\sin x}, \hspace{3mm} \tan \left(\frac{x}{2}\right)= \frac{1-\cos x}{\sin x}.
\end{equation*}
To go from $(\ref{ineq3})$ to $(\ref{ineq4})$ we use bound on $g(L), g(M-1)$ in $(\ref{relation})$ along with the fact that on $(0, \pi/2)$, $\cot x$ is non-negative and decreasing function and $\tan x$ is non-negative and increasing function.
For inequality $(\ref{ineq4})$ to $(\ref{ineq5})$ we use the relation below valid for  $\theta\neq k\pi$ and $k\in \mathbb{Z}$:
\begin{equation*}
    \tan\left(\frac{\pi}{2}-\theta\right)=\cot\left(\theta\right).
\end{equation*}
And lastly inequality $(\ref{ineq6})$ follows from $(\ref{ineq5})$ from the fact that $\cot x\leq 1/x$ for $0<x<\pi/2$.
\end{proof} 

Furthermore, for an alternate proof of Lemma \ref{incorrectevery} one can refer to  \textup{\cite[Lemma 6.6]{tenanbaum}} albeit we note a couple of typos in that proof \footnote[2]{These errors were also pointed out by Kevin Ford in \cite{xyz}}, the estimate on $|c_n|$ in the beginning of the proof should include an equality because of equality in (\ref{cotcos}) and read:
\begin{equation*}
|c_n|=|1-c_n|=\frac{1}{2|\sin \pi y_n|}\leq \frac{1}{2\sin \pi \vartheta}
\end{equation*}
 Next,  the inequality at the end of that proof should involve a negative sign between cotangent terms and read:
\begin{equation}
    \left\lvert\sum_{1\leq n \leq N}e(x_n)\right\rvert\leq \frac{1}{2}\cot(\pi y_1)-\frac{1}{2}\cot (\pi y_{N-1})+|c_1|+|1-c_{N-1}|
\end{equation}

Moreover, Landau showed in \cite{landauineq} that the constant $2/\pi$ in $(\ref{correct})$ is the best possible. For historical context behind this result we refer to Reyna's work in \cite{reynacorrection} where additionally Reyna gives an alternate proof of Lemma \ref{incorrectevery} above. Although we make note of a typo in \textup{\cite[Lemma 2(a)]{reynacorrection}}, where the inequality on $b_k$ should read: $\pi \theta\leq b_k\leq \pi(1-\theta)$.

Next, having corrected Lemma 2 in \cite{graham} we now correct Lemma 3 in \cite{graham} which is a crucial tool used in literature to obtain explicit estimates.
\begin{lemma}
Assume that $f$ is a real-valued function with two consecutive derivatives on $[N+1, N+L].$ If there exists two real numbers $V<W$ with $W>1$ such that
\begin{equation}
    \frac{1}{W}\leq |f''(x)|\leq \frac{1}{V} \label{seconddevbound}
\end{equation}
for $x$ on $[N+1, N+L],$ then 
\begin{equation}
    \left\lvert\sum_{n=N+1}^{N+L}e^{2\pi i f(n)}\right\rvert \leq 2\left(\frac{L}{V}+2\right)\left(2\sqrt{\frac{W}{\pi}}+1\right) \label{correctchenggraham}
\end{equation}
\end{lemma}
\begin{remark}
Because of this correction, the leading term in the incorrect Cheng-Graham Lemma 3 is off by a factor of $\sqrt{2}$.
\end{remark}
\begin{proof}
The proof in \textup{\cite[Lemma 3]{graham}} is modified as follows: The $k-1$ sub-sums corresponding to the interval $[y_j, x_{j+1}]$ are bounded by $\frac{2}{\pi \Delta}$ instead of $1/(\pi \Delta)+1$. And we take $\Delta=1/\sqrt{\pi W}$ instead of $1/\sqrt{2\pi W}$. We also make note of a couple of typos in the proof: First, the estimate on $k$ should be $k\leq L/V+3$. Next, when estimating the sum trivially, the mean value theorem in this case should be applied to $(f')^{-1}$ instead of $f^{-1}.$
\end{proof}
Moreover using Platt-Trudgian's observation in  \textup{\cite[Lemma 1]{platt}} we obtain a slight improvement to $(\ref{correctchenggraham})$ in the form of
\begin{equation}
    \left\lvert\sum_{n=N+1}^{N+L}e^{2\pi i f(n)}\right\rvert \leq 2\left(\frac{L-1}{V}+2\right)\left(2\sqrt{\frac{W}{\pi}}+\frac{1}{2}\right)+1
\end{equation}
The corrected version $(\ref{correctchenggraham})$ in turn gives us new constants in the explicit third dervivative test found in \textup{\cite[Lemma 1.2]{hiary}} and again for completeness we state the corrected version here.
\begin{lemma}\label{derivative3test}\footnote[3]{This corrected version of the explicit third derivative test in turn changes the explicit van der Corput bound derived by Hiary in \textup{\cite[Theorem 1.1]{hiary}}. In that paper, Hiary obtained $|\zeta(1/2+it)|\leq 0.63t^{1/6}\log t$ at the time. This bound was an improvement to Platt-Trudgian's result in \textup{\cite[Theorem 1]{platt}} that stated $|\zeta(1/2+it)|\leq 0.732 t^{1/6}\log t$. With the correction provided in Lemma \ref{derivative3test}, the stimate obtained by Hiary in \cite{hiary} now becomes $|\zeta(1/2+it)|\leq 0.77t^{1/6}\log t$ (with $t_0=2\times 10^{10}$ in that paper). And since the Platt-Trudgian bound for $|\zeta(1/2+it)|$ uses the incorrect Cheng-Graham lemma, it may no longer be valid. 

Nonetheless the constant, 0.77, is currently being improved by the author and the result will be published soon along with some additional estimates for $|\zeta(1/2+it)|$.}
Let $f(x)$ be a real-valued function with four continuous derivatives on ${[N+1, N+L]}$. Suppose there are constants $W_3>1$ and $\lambda_3\geq 1$ such that $\frac{1}{W_3}\leq |f^{(3)}(x)|\leq \frac{\lambda_3}{W_3}$ for ${N+1\leq x\leq N+L}$. If $\eta_3>0,$ then
\begin{equation*}
    \left\lvert\sum_{n=N+1}^{N+L}e^{2\pi i f(n)}\right\rvert^2 \leq (LW_3^{-1/3}+\eta_3)(\alpha_3 L+\beta_3 W_3^{2/3}),
\end{equation*}
where
\begin{align*}
\alpha_3 &=\frac{1}{\eta_3}+\frac{32\lambda}{15\sqrt{\pi}}\sqrt{\eta_3+W_3^{-1/3}}+\frac{2\lambda_3\eta_3}{W_3^{1/3}}+\frac{2\lambda_3}{W_3^{2/3}},    \\
\beta_3 &= \frac{64}{3\sqrt{\pi}\sqrt{\eta_3}}+\frac{4}{W_3^{1/3}}.
\end{align*}
\end{lemma}
\begin{proof}
To obtain this corrected version, we use $(\ref{correctchenggraham})$ above in the proof of Lemma 1.2 in \cite{hiary} and replace the estmiate for $|S_m'(L)|$ in equation (43) of that proof with 
\begin{equation}
    |S_m'(L)|\leq \frac{4\lambda_3 L\sqrt{m/W_3}}{\sqrt{\pi}}+\frac{2\lambda_3 Lm}{W_3}+\frac{8\sqrt{W_3/m}}{\sqrt{\pi}}+4.
\end{equation}
\end{proof}

We note that here since $\alpha_3, \beta_3$ are decreasing function in $W_3$ and $W_3>1$ we get that the estimates:
\begin{align*}
\alpha_3 \leq \widetilde{\alpha_{3}}=\frac{1}{\eta_3}+\frac{32\lambda}{15\sqrt{\pi}}\sqrt{\eta_3+1}+2\lambda_3\eta_3+ 2\lambda_3,  \hspace{4mm}  
\beta_3 &\leq \widetilde{\beta_{3}}= \frac{64}{3\sqrt{\pi}\sqrt{\eta_3}}+4.
\end{align*}

Note that we will obtain our explicit result for $k$ derivative test by using the $k-1$ test. Hence, for clarity we label the constants and other quantities in these results in such a way that the subscripts indicate the derivative test they arise from. With Lemma $\ref{derivative3test}$ at hand we are now ready to find explicit versions of fourth and fifth derivative tests.  

\section{Preliminary Results}
\begin{lemma}\label{derivative4}
Let $f(x)$ be a real-valued function with four continuous derivatives on ${[N+1, N+L]}$. Suppose there are constants $W_4>1$ and $\lambda_4\geq 1$ such that $\frac{1}{W_4}\leq |f^{(4)}(x)|\leq \frac{\lambda_4}{W_4}$ for ${N+1\leq x\leq N+L}$. If $\eta_4>0,$ then
\begin{equation*}
    |S(L)|^2\coloneqq\left\lvert\sum_{n=N+1}^{N+L}e^{2\pi i f(n)}\right\rvert^2 \leq (LW_4^{-1/7}+\eta_4)(\alpha_4 L+\gamma_4\sqrt{L}W_4^{2/7}+\beta_4 W_4^{3/7}),
\end{equation*}
where
\begin{align*}
\alpha_4 &=\frac{1}{\eta_4}+\frac{72}{91}\sqrt{\widetilde\alpha_{3}}(\eta_4+W_4^{-1/7})^{1/6}, \hspace{3mm}
\gamma_4 = \frac{72}{55}\sqrt{\widetilde\beta_{3}}\eta_4^{-1/6}+\sqrt{\eta_3\widetilde\alpha_{3}}W_4^{-1/7},\hspace{3mm}
\beta_4 = \frac{18}{20}\sqrt{\widetilde\beta_{3}\eta_3}\eta_4^{-1/3}.
\end{align*}
\end{lemma}

\begin{proof}
We will use the Weyl-van der Corput Lemma in Cheng and Graham \textup{\cite[Lemma 5]{graham}}, but use the form given at the bottom of page 1273 as well as a further refinement by Platt and Trudgian \textup{\cite[Lemma 2]{platt}}. In all, if $M$ is a positive integer, then 
\begin{equation}
\lvert S(L)\rvert^2 \leq (L+M-1)\left(\frac{L}{M}+\frac{2}{M}\sum_{m=1}^{M}\left(1-\frac{m}{M}\right)\lvert{S_m'(L)}\rvert\right), \label{s14}
\end{equation}
where 
\begin{equation}
    S_m'(L)=\sum_{r=N+1}^{N+L-m}e^{2\pi i (f(r+m)-f(r))}.
\end{equation}
Here, we can assume that $m<L$ and $L>1.$ Otherwise, the sum $S_m'(L)$ is empty and does not contribute to the upper bound.

Now, let $g(x)\coloneqq f(x+m)-f(x)$ where $N+1\leq x\leq N+L-m$. Then
${g'''(x)=f'''(x+m)-f'''(x)}.$
Hence, using the mean value theorem we obtain
$g'''(x)=mf^{(4)}(\xi)$
for some \\
${\xi\in (x, x+m)\subset[N+1, N+L]}$. Next, using the given bound on $f^{(4)}(x)$, we deduce that 

\begin{equation}
\frac{m}{W_4}\leq\lvert g'''(x)\rvert \leq \frac{m\lambda_4}{W_4}, \hspace{5mm} \mbox{$(N+1\leq x\leq N+L-m)$}.
\end{equation}

\noindent
Applying Lemma $\ref{derivative3test}$ to bound $\lvert S_m'(L)\rvert^2$, multiplying the terms out and then using the inequalities:
\begin{align}
    \sqrt{x_1+x_2+\ldots}&\leq \sqrt{x_1}+\sqrt{x_2}+\ldots, \hspace{8mm}\text{where $x_1, x_2, \ldots>0$} \nonumber\\
    \alpha_3&\leq \widetilde\alpha_{3}\hspace{2mm} \text{and} \hspace{2mm}\beta_3\leq \widetilde\beta_{3}, \hspace{3mm} L-m< L\hspace{2mm} \text{for} \hspace{2mm} m\geq 1. \label{impinequalities}
\end{align}
\begin{align}
   \lvert S_m'(L)\rvert^2 < \sqrt{\widetilde\alpha_{3}}LW_4^{-1/6}m^{1/6}+\sqrt{\widetilde\beta_{3} L}W_4^{1/6}m^{-1/6} &+\sqrt{\eta_3\widetilde\alpha_{3}L}+\sqrt{\widetilde\beta_{3}\eta_3}W_4^{1/3}m^{-1/3}.\label{b14}
\end{align}
Next, let us bound $\displaystyle\sum_{m=1}^{M}\left(1-\frac{m}{M}\right) \lvert S_m'(L)\rvert$ using (\ref{b14}) and the below estimate valid for $-1<q<1$:
\begin{equation}
    \sum_{m=1}^{M}\left(1-\frac{m}{M}\right)m^{q} \leq \frac{M^{q+1}}{(q+1)(q+2)}. \label{eulerineq}
\end{equation}

To prove the above estimate for $q\geq 0$ we use \textup{\cite[Lemma 7]{graham}} and for $q<0$ we replace the sum with an integral. After obtaining such bounds we substitute them back in ($\ref{s14}$) and get 
\begin{align}
\begin{split}\label{s24}
   \lvert S(L)\rvert^2 \leq {}& (L+M-1)\Bigg(\frac{L}{M}+\frac{72}{91}\sqrt{\widetilde\alpha_{3}}LW_4^{-1/6}M^{1/6}+\frac{72}{55}\sqrt{\widetilde\beta_{3} L}W_4^{1/6}M^{-1/6}\\
         & \hspace{40mm}+\sqrt{\eta_3\widetilde\alpha_{3} L}+\frac{18}{10}\sqrt{{\widetilde\beta_{3}}\eta_3}W_4^{1/3}M^{-1/3}\Bigg).
\end{split}
\end{align}
Now we would like to make the first two terms in ($\ref{s24}$) of the same magnitude to minimize the rhs. This can be achieved if we choose $M =\lceil\eta_4 W_4^{1/7}\rceil$ for some free parameter $\eta_4>0$ that can be optimized. With this choice of $M$ we obtain the inequality $\eta_4 W^{1/7}\leq M\leq \eta_4 W^{1/7} +1$. Using this inequality and then factoring $W^{1/7}$ term from the first parenthesis and multiplying it in the second we deduce that
\begin{align}
\lvert S(L)\rvert^2 &\leq  (LW_4^{-1/7}+\eta_4)W_4^{1/7}\Bigg(\Bigg({\eta_4^{-1} W_4^{-1/7}}+\frac{72}{91}\sqrt{\widetilde\alpha_{3}}\left(\frac{\eta_4 W_4^{1/7}+1}{W_4}\right)^{1/6}\Bigg)L\\
     &+\Bigg(\frac{72}{55}\sqrt{\widetilde\beta_{3}}\eta_4^{-1/6} W_4^{1/7}+\sqrt{\eta_3\widetilde\alpha_{3}}\Bigg)\sqrt{L}+\frac{18}{10}\sqrt{{\widetilde\beta_{3}}\eta_3}\eta_4^{-1/3} W_4^{6/21}\Bigg)\\
 &\leq (LW_4^{-1/7}+\eta_4)(\alpha_4 L+\gamma_4\sqrt{L}W_4^{2/7}+\beta_4 W_4^{3/7}) \label{f14}
\end{align}
where
$\alpha_4, \gamma_4$, and $\beta_4$ are defined as in the statement of the lemma. 
\end{proof}

\begin{lemma}\label{fifthderivative}
Let $f(x)$ be a real-valued function with five continuous derivatives on ${[N+1, N+L]}$. Suppose there are constants $W_5>1$ and $\lambda_5\geq 1$ such that $\frac{1}{W_5}\leq |f^{(5)}(x)|\leq \frac{\lambda_5}{W_5}$ for ${N+1\leq x\leq N+L}$. If $\eta_5>0,$ then
\begin{equation*}
    \left\lvert\sum_{n=N+1}^{N+L}e^{2\pi i f(n)}\right\rvert^2 \leq (LW_5^{-1/15}+\eta_5)(\alpha_5 L+\tau_5L^{3/4}W_5^{2/15}+\gamma_5\sqrt{L}W_5^{3/15}+\omega_5L^{1/4}W_5^{1/5}+\beta_5 W_5^{4/15});
\end{equation*}
where
\begin{align*}
\alpha_5 &=\frac{1}{\eta_5}+\frac{392}{435}\sqrt{\widetilde\alpha_4}\left(\eta_5+W_5^{-1/15}\right)^{1/14}, \hspace{3mm}
\tau_5 =\frac{392}{351}\sqrt{\widetilde\gamma_4}\eta_5^{-1/14},\hspace{3mm}
\gamma_5 = \frac{98}{78}\sqrt{\beta_4}\eta_5^{-1/7} +\sqrt{\eta_4 \widetilde\alpha_4 }W_5^{-2/15},\\
\omega_5 &= \frac{98}{78}\sqrt{\eta_4\widetilde\gamma_4}\eta_5^{-1/7},\hspace{3mm}
\beta_5 = \frac{392}{275}\sqrt{{\eta_4}\beta_4}\eta_5^{-3/14}.
\end{align*}
\end{lemma}
\begin{proof}
The proof of this lemma is very similar to Lemma $\ref{derivative4}$ where we first bound $|S_m'(L)|$ by letting $g(x)\coloneqq f(x+m)-f(x)$ where $N+1\leq x\leq N+L-m$. Then ${g^{(4)}(x)=f^{(4)}(x+m)-f^{(4)}(x)}$ from which we can deduce using the mean value theorem and given bound on $f^{(5)}(x)$ that 
\begin{equation}
\frac{m}{W_5}\leq\lvert g^{(4)}(x)\rvert \leq \frac{m\lambda_5}{W_5}, \hspace{5mm} \mbox{$(N+1\leq x\leq N+L-m)$}.
\end{equation}
Applying the result for fourth derivative stated in Lemma $\ref{derivative4}$ along with the estimates
\begin{equation} 
\alpha_4 \leq \widetilde{\alpha_4}=\frac{1}{\eta_4}+\frac{72}{91}\sqrt{\widetilde\alpha_{3}}(\eta_4+1)^{1/6},   \hspace{5mm}
\gamma_4 \leq \widetilde{\gamma_4}=\frac{72}{55}\sqrt{\widetilde\beta_{3}}\eta_4^{-1/6}+\sqrt{\eta_3\widetilde\alpha_{3}}
\end{equation}
 and inequalities similar to $(\ref{impinequalities})$  to bound $\lvert S_m'\rvert^2$, we get
\begin{align}
   \left\lvert S_m'(L)\right\rvert&\leq \sqrt{\widetilde{\alpha_4}}LW_5^{-1/14}m^{1/14}+\sqrt{\widetilde{\gamma_4}}L^{3/4}W_5^{1/14}m^{-1/14}+\sqrt{\beta_4 L}W_5^{1/7}m^{-1/7} \nonumber\\
   &+\sqrt{\eta_4\widetilde{\alpha_4} L}+\sqrt{\eta_4\widetilde{\gamma_4}}L^{1/4}W_5^{1/7}m^{-1/7}+\sqrt{\eta_4 \beta_4}W_5^{3/14}m^{-3/14}. \label{b1}
\end{align}
Now we bound $\displaystyle\sum_{m=1}^{M}\left(1-\frac{m}{M}\right)\left\lvert\sum_{r=N+1}^{N+L-m}e^{2\pi i (g(r))}\right\rvert$ using ($\ref{b1}$) and $(\ref{eulerineq})$ and then substitute them  in an expression like $(\ref{s14})$ to get 
\begin{align}
    \Bigg\lvert{\sum_{n=N+1}^{N+L}e^{2\pi i f(n)}}\Bigg\rvert^2&\leq (L+M-1)\Bigg(\frac{L}{M}+\frac{392}{435}\sqrt{\widetilde\alpha_4}LW_5^{-1/14}M^{1/14}+\frac{392}{351}\sqrt{\widetilde\gamma_4}L^{3/4}W_5^{1/14}M^{-1/14}\nonumber\\
     &\hspace{20mm}+\frac{98}{78}\sqrt{\beta_4L}W_5^{1/7}M^{-1/7}+\sqrt{\eta_4 \widetilde\alpha_4 L}+\frac{98}{78}\sqrt{\eta_4\widetilde\gamma_4}L^{1/4}W_5^{1/7}M^{-1/7}\nonumber\\
     &\hspace{20mm}+\frac{392}{275}\sqrt{{\eta_4}\beta_4}W_5^{3/14}M^{-3/14}\Bigg) \label{s2}
\end{align}
Now we would like to make the first two terms in ($\ref{s2}$) of the same magnitude to minimize the rhs in $(\ref{s2})$. This can be achieved if we choose $M =\lceil\eta_5 W_5^{1/15}\rceil$ for some free parameter $\eta_5>0$ that can be optimized. Next, using the inequality $\eta_5 W^{1/15}\leq M\leq \eta_5 W_5^{1/15} +1$ obtained because of the choice of $M$ and using similar algebraic manipulations as in proof of Lemma $\ref{derivative4}$ we deduce that
\begin{align}
    \left\lvert{\sum_{n=N+1}^{N+L}e^{2\pi i f(n)}}\right\rvert^2 &\leq  (LW_5^{-1/15}+\eta_5)\Bigg(\Bigg({\eta_5^{-1} }+\frac{392}{435}\sqrt{\widetilde\alpha_4}\left(\eta_5+W_5^{-1/15}\right)^{1/14}\Bigg)L\nonumber\\
     &+\frac{392}{351}\sqrt{\widetilde\gamma_4}\eta_5^{-1/14} L^{3/4}W_5^{2/15}+\Bigg(\frac{98}{78}\sqrt{\beta_4}\eta_5^{-1/7} +\sqrt{\eta_4 \widetilde\alpha_4}W_5^{-2/15}\Bigg)\sqrt{L}W_5^{1/5}\nonumber\\
     &+\frac{98}{78}\sqrt{\eta_4\widetilde\gamma_4}\eta_5^{-1/7}L^{1/4}W_5^{1/5}+\frac{392}{275}\sqrt{{\eta_4}\beta_4}\eta_5^{-3/14}W_5^{4/15}\Bigg).
\end{align}
This finally gives us:
\begin{equation}
    \left\lvert\sum_{n=N+1}^{N+L}e^{2\pi i f(n)}\right\rvert^2 \leq (LW_5^{-1/15}+\eta_5)(\alpha_5 L+\tau_5L^{3/4}W_5^{2/15}+\gamma_5\sqrt{L}W_5^{3/15}+\omega_5L^{1/4}W_5^{1/5}+\beta_5 W_5^{4/15}) \label{f1}
\end{equation}
where $\alpha_5, \tau_5, \gamma_5, \omega_5$ and $\beta_5$ are defined as in the statement of the lemma.
\end{proof}
Note that Lemmas \ref{derivative3test}, \ref{derivative4} and \ref{fifthderivative} are explicit versions of processes $AB, A^2B, A^3B$ in the theory of exponent pairs respectively. For an introduction to the theory of exponent pairs the author refers the reader to \cite{grahambook}. Moreover these lemmas give a saving of $\approx W^{n}$ when compared to the trivial bound where $n=\displaystyle-\frac{1}{3}, -\frac{1}{7}$ and $-\displaystyle\frac{1}{15}$ for Lemmas $\ref{derivative3test}, \ref{derivative4}$ and $\ref{fifthderivative}$. In application, it is often unclear on how to choose the correct derivative test to obtain an estimate. For instance as it will be seen later, the choice of $W$, dictates the length of the interval over which the $k$-th derivative test is applied while bounding the zeta function. This along with the method used to estimate the initial sum determines the derivative test to be applied. 

\begin{lemma}\label{gammabound}
If $s=\sigma+it$ where $\sigma>0$ and $t>0 $ then we have
\begin{equation}
   \sqrt{2\pi}e^{-\pi t/2}\xi_1(\sigma, t)t^{\sigma-1/2} \leq |\Gamma(s)|\leq \sqrt{2\pi}e^{-\pi t/2}\xi_2(\sigma, t)t^{\sigma-1/2} \label{gammaexpbound}
\end{equation}
where 
\begin{align*}
\xi_1(\sigma, t)&=\exp\left(-\frac{\sigma}{12t^2}-\frac{\pi}{24t}-\frac{\sigma^3}{3t^2}\right),\\
\xi_2(\sigma, t)&=\exp\left({\left(\sigma-\frac{1}{2}\right)\frac{\sigma^2}{2t^2}+\frac{\sigma}{12t^2}+\frac{\pi}{24t}}\right).
\end{align*}
\end{lemma}
\begin{proof}
The proof follows a similar strategy as in \cite{habsieger}. By Stirling's formula for complex values $s$ such that $-\pi+\delta\leq \arg s\leq \pi-\delta$ given in \textup{\cite[Page 151]{titchmarshfuncbook}} we have 
\begin{align}
    \log\Gamma(s)=\left(s-\frac{1}{2}\right)\log s-s+\frac{\log(2\pi)}{2}-\int_{0}^{\infty}\frac{\{x\}-1/2}{x+s}\, dx.
\end{align}
This gives
\begin{align}
    \log|\Gamma(s)|=\left(\sigma-\frac{1}{2}\right)\log |s|+t\left(-\arctan\left(\frac{t}{\sigma}\right)-\frac{\sigma}{t}\right)+\frac{\log(2\pi)}{2}-\Re\left(\int_{0}^{\infty}\frac{\{x\}-1/2}{x+s}\, dx\right) \label{gamma}.
\end{align}
Using integration by parts we have 
\begin{align*}
    \int_{0}^{\infty}\frac{\{x\}-1/2}{x+s}\, dx &= -\frac{1}{12s}+\frac{1}{2}\int_{0}^{\infty}\frac{\{x\}^2-\{x\}+1/6}{(x+s)^2}
\end{align*}
Hence we have using triangle inequality and $\sigma\geq 0$
\begin{align}
   \left \lvert\Re\left(\int_{0}^{\infty}\frac{\{x\}-1/2}{x+s}\, dx\right)\right\rvert&\leq \frac{\sigma}{12|s|^2}+\frac{1}{12}\int_{0}^{\infty}\frac{dx}{(x+\sigma)^2+t^2}\leq \frac{\sigma}{12t^2}+\frac{\pi}{24t}\label{repart}
\end{align}
Also, since $0\leq \log(1+x)\leq x$ for $x\geq 0$ we also have
\begin{align}
    0\leq \log|s|-\log t=\frac{1}{2}\log\left(1+\frac{\sigma^2}{t^2}\right)\leq \frac{\sigma^2}{2t^2}, \hspace{5mm} \text{for}\hspace{2mm} \sigma\geq 0 \label{logineq}
\end{align}
Moreover note that since $\sigma>0$
\begin{equation}
    -\frac{\pi}{2}-\frac{\sigma^3}{3t^3}\leq-\arctan\left(\frac{t}{\sigma}\right)-\frac{\sigma}{t}=\arctan\left(\frac{\sigma}{t}\right)-\frac{\pi}{2}-\frac{\sigma}{t}\leq -\frac{\pi}{2} \label{arctanineq}
\end{equation}
where the inequalities follows from for $x\geq 0$
\begin{equation*}
    x-\frac{1}{3}x^3\leq\arctan(x)\leq x
\end{equation*}
and the equality in the middle is because
\begin{align*}
\arctan(x)+\arctan(1/x)=\frac{\pi}{2}, \hspace{5mm} \text{for}\hspace{2mm} x>0
\end{align*}
Thus substituting $(\ref{logineq}), (\ref{arctanineq}), (\ref{repart})$ in $(\ref{gamma})$ we obtain
\begin{align}
    \log|\Gamma(s)|&\leq \log\sqrt{2\pi}+\left(\sigma-\frac{1}{2}\right)\log t-\frac{\pi}{2}t+\left(\sigma-\frac{1}{2}\right)\frac{\sigma^2}{2t^2}+\frac{\sigma}{12t^2}+\frac{\pi}{24t}\\
    \log|\Gamma(s)|&\geq \frac{\log(2\pi)}{2}+\left(\sigma-\frac{1}{2}\right)\log t-\frac{\pi}{2}t-\frac{\sigma^3}{3t^2}-\frac{\sigma}{12t^2}-\frac{\pi}{24t}
\end{align}
Therefore exponentiating on both sides we get
\begin{equation*}
    \sqrt{2\pi}e^{-\pi t/2}\xi_1(\sigma, t)t^{\sigma-1/2}\leq\Gamma(s)\leq \sqrt{2\pi}e^{-\pi t_/2}\xi_2(\sigma, t)t^{\sigma-1/2}
\end{equation*}
where 
\begin{align}
\xi_1(\sigma, t)&=\exp\left(-\frac{\sigma}{12t^2}-\frac{\pi}{24t}-\frac{\sigma^3}{3t^2}\right)\\
\xi_2(\sigma, t)&=\exp\left({\left(\sigma-\frac{1}{2}\right)\frac{\sigma^2}{2t^2}+\frac{\sigma}{12t^2}+\frac{\pi}{24t}}\right)
\end{align}
Here we remark that for $\sigma>0, t\geq t_0>0$ 
\begin{equation}
\xi_1(\sigma, t)\geq \xi_1(\sigma, t_0), \hspace{5mm} 
\label{constantb} \xi_2(\sigma, t) \leq
\begin{cases}
 \vspace{5mm}
\displaystyle\frac{\sigma}{12t_0^2}+\frac{\pi}{24t_0} & \text{for}\ 0< \sigma< \frac{1}{2}\\
 \displaystyle\xi_2(\sigma, t_0) &\text{for}\ \sigma\geq \frac{1}{2}
\end{cases}
\end{equation}
\end{proof}
We also remark that $(\ref{gammaexpbound})$ is also valid for $t<0$ provided that $t$ is replace with $|t|$ since $\Gamma(\overline{s})=\overline{\Gamma(s)}$ and $|\Gamma(\overline{s})|=|\overline{\Gamma(s)}|$

\begin{corollary}\label{gengammabound}
We have for $\sigma\leq 0$ and $t>0$
\begin{equation*}
    \xi_3(\sigma, t, n)|\Gamma((\sigma+n)+it)|\leq |\Gamma(\sigma+it)|\leq\frac{1}{t^n}|\Gamma((\sigma+n)+it)|
\end{equation*}
where 
$n$ is the smallest positive integer such that $\Re(\sigma+n)>0$ and
\begin{equation*}
    \xi_3(\sigma, t, n)=\frac{1}{(-\sigma+t)(-\sigma-1+t)\ldots(-\sigma-n+1+t)}.
\end{equation*}
\end{corollary}
\begin{proof}
The functional equation of gamma function states
\begin{equation}
    \Gamma(z)=\frac{1}{z}\Gamma(z+1), \hspace{8mm} z\in \mathbb{C}\setminus \{0, -1, -2, \ldots\} \label{gammafunceq}
\end{equation}
Now let $n\in \mathbb{Z}_{>0}$ be the smallest integer such that $\Re(z+n)>0$ Thus using $(\ref{gammafunceq})$ we obtain the following relation:
\begin{align*}
    |\Gamma(z)|=\frac{1}{|z(z+1)(z+2)\cdots(z+n-1)|}|\Gamma(z+n)|.
\end{align*}
Hence substituting $z=\sigma+it$ and using Lemma \ref{gammabound} with $\sqrt{(\sigma+l)^2+t^2}\geq t$ for $l \in \{0, 1, \ldots, n-1\}$ we get that 
\begin{align*}
    |\Gamma(\sigma+it)|\leq\frac{1}{t^n}|\Gamma(\sigma+n+it)|
\end{align*}
And since $\sigma+l<0,$ we can find a lower bound using $\sqrt{(\sigma+l)^2+t^2}\leq-(\sigma+l)+t$, for $l\in \{0, 1, \ldots n-1\}$ and thus
\begin{align*}
    |\Gamma(\sigma+it)|&\geq \frac{1}{(-\sigma+t)(-\sigma-1+t)\ldots(-\sigma-n+1+t)}|\Gamma((\sigma+n)+it)|.
\end{align*}
Hence the result follows.
\end{proof}

\begin{corollary}\label{chiexp1}
For $t>0$ we have
\begin{equation}
    |\chi(1+it)|\leq \frac{g(t)}{t^{1/2}}
\end{equation}
where 
\begin{equation}
    g(t)=\sqrt{2\pi}\exp\left(\frac{5}{3t^2}+\frac{\pi}{6t}\right) \label{gfunc}
\end{equation}
\end{corollary}
\begin{proof}
We have the following definition for $\chi(1+it)$ as stated in \textup{\cite[Page 16]{titchmarsh}} when $s=1+it$ is substituted
\begin{equation}
    \chi(1+it)=\pi^{1/2+it}\frac{\Gamma\left(\frac{-it}{2}\right)}{\Gamma\left(\frac{1+it}{2}\right)}. \label{reynachi}
\end{equation}

In order to bound $|\chi(1+it)|$ using $(\ref{reynachi})$ we will first bound $\displaystyle\left|\Gamma\left(\frac{-it}{2}\right)\right|$ from above and $\displaystyle\left|\Gamma\left(\frac{1+it}{2}\right)\right|$ from below.
First using Corollary $\ref{gengammabound}$ and then Lemma $\ref{gammabound}$ along with the remark following it we have that 
\begin{equation}
   \Bigg|\Gamma\left(-\frac{it}{2}\right)\Bigg| \leq \frac{2\sqrt{\pi}}{t^{1/2}}e^{-\pi t/4}\xi_2(1, t/2), \hspace{5mm} \Bigg|\Gamma\left(\frac{1+it}{2}\right)\Bigg|\geq \sqrt{2\pi}e^{-\pi t/4}\xi_1(1/2, t/2) \label{expgamma}
\end{equation}
which in turn gives us the result 
\begin{align*}
    |\chi(1+it)|\leq \sqrt{2\pi}\exp\left(\frac{5}{3t^2}+\frac{\pi}{6t}\right)\frac{1}{t^{1/2}}
\end{align*}
where we used $n=1$ and
\begin{align}
    \Bigg|\frac{\xi_2(1, t/2)}{\xi_1(1/2, t/2)}\Bigg|\leq \frac{\exp\left(\displaystyle\frac{4}{3t^2}+\displaystyle\frac{\pi}{12t}\right)}{\exp\left(-\displaystyle\frac{1}{3t^2}-\displaystyle\frac{\pi}{12t}\right)}=\exp\left(\displaystyle\frac{5}{3t^2}+\frac{\pi}{6t}\right), \hspace{5mm} \text{for}\hspace{2mm} t>0
\end{align}
\end{proof}

\begin{theorem} \label{zeta1exp}
For $t>0$ and $n_1=\lfloor\sqrt{t/2\pi}\rfloor$ we have
\begin{equation}
  |\zeta(1+it)|\leq  \left\lvert \sum_{n=1}^{n_1}\frac{1}{n^{1+it}}\right\rvert+\frac{g(t)}{t^{1/2}}\left\lvert \sum_{n=1}^{n_1}\frac{1}{n^{-it}}\right\rvert+\mathcal{R} \label{main1linesum}
\end{equation}
where 
\begin{equation*}
    \mathcal{R}=\mathcal{R}(t)\coloneqq\frac{1}{t^{1/2}}\left(\sqrt{\frac{\pi}{2}}+\frac{g(t)}{2}\right)+\frac{1}{t}\left(9\sqrt{\frac{\pi}{2}}+\frac{g(t)}{\sqrt{\pi(3-2\log 2)}}\right)+\frac{1}{t^{3/2}}\left(\frac{968\pi^{3/2}+g(t)242\pi}{700}\right).
\end{equation*}
and $g(t)$ is  given by $(\ref{gfunc})$.
\end{theorem}
\begin{proof}
Note that Siegel had obtained the following expression for $\zeta(s):$
\begin{equation}
 \zeta(s)=\mathcal{R}(s)+\chi(s)\overline{\mathcal{R}}(1-s),  \hspace{10mm}\label{maint}
\end{equation}
where $\mathcal{R}(s)$ is defined by an integral as given in \cite{Reyna} and
\begin{equation*}
   \overline{\mathcal{R}}(s)=\overline{\mathcal{R}(\overline{s})}, \hspace{4mm} \chi(s)=\pi^{s-1/2}\frac{\Gamma(\frac{1-s}{2})}{\Gamma(\frac{s}{2})}
\end{equation*}
Reyna showed in \cite{Reyna} that with $\sigma$ and $t$ real and $t>0$, and an integer $K\geq 0$ we have
\begin{align}
    \mathcal{R}(s)=\sum_{n=1}^{N}\frac{1}{n^s}+(-1)^{N-1}Ua^{-\sigma}\left\{\sum_{k=0}^K\frac{C_k(p)}{a^k}+RS_K\right\} \label{partialsum}
\end{align}
where 
\begin{align*}
    a&\coloneqq\sqrt{\frac{t}{2\pi}}, \hspace{5mm}N\coloneqq \lfloor a\rfloor,\hspace{5mm} p\coloneqq 1-2(a-N)\\
    U&\coloneqq \exp\left(-i\left(\frac{t}{2}\log\frac{t}{2\pi}-\frac{t}{2}-\frac{\pi}{8}\right)\right)
\end{align*}
and $RS_K$ and $C_k(p)$ are defined in \textup{\cite[Page 999]{Reyna}}. 

In particular substituting $s=1+it$ in $(\ref{maint})$ we obtain
\begin{equation}
    |\zeta(1+it)|\leq |\mathcal{R}(1+it)|+|\chi(1+it)||\overline{\mathcal{R}}(-it)| \label{1earlybound}
\end{equation}
Thus to bound $|\zeta(1+it)|$ we will first bound $|\mathcal{R}(1+it)|$, $|\overline{\mathcal{R}}(-it)|$ using $(\ref{partialsum})$, triangle inequality and \cite{Reyna} and combine it with the bound $|\chi(1+it)|$ found in Corollary $\ref{chiexp1}$.
Hence, after using $(\ref{partialsum})$ and triangle inequality we get
\begin{align}
   |\mathcal{R}(1+it)|&\leq \left\lvert \sum_{n=1}^{N}\frac{1}{n^{1+it}}\right\rvert+\left(\frac{2\pi}{t}\right)^{1/2}\left(\left\lvert\sum_{k=0}^{K}\frac{C_k(p)}{a^k}\right\rvert+\left\lvert RS_K\right\rvert\right). \label{r}
\end{align}
It remains to bound the last two terms to the right of the inequality above. First, we will bound the second sum on the right hand side of $(\ref{r})$ using $K=1$, triangle inequality after explanding the sum and \textup{\cite[Theorem 4.1]{Reyna}} with $\sigma=1>0,$ and $\Gamma(1/2)=\sqrt{\pi}$. We thus have
\begin{align}
    \left\lvert\sum_{k=0}^1\frac{C_k(p)}{a^k}\right\rvert\leq \left\lvert C_0(p)\right\rvert +\frac{9}{2t^{1/2}}\leq \frac{1}{2}+\frac{9}{2t^{1/2}}. \label{r0}
\end{align}
The last inequality above follows from bounding $|C_0(p)|$ using $(5.2)$ and Theorem 6.1 in \cite{Reyna}.

Next to bound $|RS_K|$ we use \textup{\cite[Theorem 4.2]{Reyna}}, again with $K=1$ and $\sigma =1, t>0$ and $\Gamma(1)=1$. So we obtain
\begin{align}
    |RS_1|\leq\frac{242\pi}{700t}2^{3/2} \label{r1}.
\end{align}
Now plugging in $(\ref{r0})$ and $(\ref{r1})$ in $(\ref{r})$ we get
\begin{align}
    |\mathcal{R}(1+it)|&\leq \left\lvert \sum_{n=1}^{N}\frac{1}{n^{1+it}}\right\rvert+\frac{1}{t^{1/2}}\left(\frac{(2\pi)^{1/2}}{2}+\frac{968\pi^{3/2} }{700t}+\frac{9(2\pi)^{1/2}}{2t^{1/2}}\right) \label{firstterm}
\end{align}
Similarly we can a bound for $|\overline{\mathcal{R}}(-it)|$ by first observing  $|\mathcal{\overline{R}}(-it)|=|\overline{\mathcal{R}(it)}|=|\mathcal{R}(it)|$ and using triangle inequality and $(\ref{partialsum})$ we get
\begin{align}
    |\mathcal{\overline{R}}(-it)|\leq \left\lvert \sum_{n=1}^{N}\frac{1}{n^{it}}\right\rvert+\left(\left\lvert \sum_{k=0}^K\frac{C_k(p)}{a^k}\right\rvert+\left\lvert RS_K\right\rvert\right) \label{initialsum2}
\end{align}

Now to bound the second sum and $|RS_K|$ above we follow similar steps used to derive $(\ref{r0})$ and $(\ref{r1})$ with $\sigma=0$ and $t>0$ and obtain
\begin{align}
    |\mathcal{\overline{R}}(-it)|\leq \left\lvert \sum_{n=1}^{N}\frac{1}{n^{it}}\right\rvert+\frac{1}{2}+\frac{1}{\sqrt{\pi(3-2\log 2)}}\frac{1}{t^{1/2}}+\frac{242\pi}{700t} \label{initialsum2}
\end{align}

Lastly combining $(\ref{maint})$, $(\ref{firstterm})$,  $(\ref{initialsum2})$, and the bound for $|\chi(1+it)|$ using Corollary $\ref{chiexp1}$ we have our theorem.
\end{proof}

\begin{lemma}[Partial Summation] \label{ps}
Let $b_1\geq b_2\geq \ldots \geq b_n\geq 0$, and $s_m=a_1+a_2+\ldots+a_m$ where the $a's$ are any real or complex numbers. Then if $|s_m|\leq M (m=1, 2, \ldots)$,
\begin{equation}
    |a_1b_1+a_2b_2+\ldots a_nb_n|\leq b_1M.
\end{equation}
\end{lemma}
This lemma can be found in \textup{\cite[Page 96]{titchmarsh}}. In this paper we will use it to  remove the $n^{-1}$ weight from the sum $\sum n^{-1 -it}$. In order to do so, we let $b_n=\frac{1}{n}$, $a_n=n^{-it}$ and $M=\displaystyle\max_{1\leq m\leq L}\left|\sum_{k=N+1}^{N+m}n^{-it}\right|$ giving us 
\begin{equation*}
    \left|\sum_{n=N+1}^{N+L}\frac{e^{-it\log n}}{n}\right|\leq \frac{1}{N+1}\max_{1\leq \Delta\leq L}\left|\sum_{N+1}^{N+\Delta}e^{-it\log n}\right|.
\end{equation*}
\section{Proof of Theorem \ref{maintheorem}}
First using Theorem $\ref{zeta1exp}$ we get a significant improvement on known explicit bounds of the form $(\ref{landaugen})$ by estimating  $(\ref{main1linesum})$ trivially using triangle inequality, $(\ref{harmonicsum})$, $(\ref{gfunc})$, and $\lfloor t^{1/2}/\sqrt{2\pi}\rfloor\geq t^{1/2}/\sqrt{2\pi}-1$ to get the following bound valid for $t>2\pi$:
\begin{align} 
    |\zeta(1+it)|&\leq \frac{1}{2}\log t+\gamma+\frac{\sqrt{2\pi}}{t^{1/2}-\sqrt{2\pi}}-\frac{1}{2}\log(2\pi)+\exp\left(\frac{5}{3t^2}+\frac{\pi}{6t}\right)+\mathcal{R} \label{zetatriangle}
\end{align}

To improve upon $(\ref{zetatriangle})$ we make use of Lemma $\ref{fifthderivative}$. With this in mind, we first split the sums on the right of the inequality in $(\ref{main1linesum})$ as follows:
\begin{align}
   |\zeta(1+it)|\leq \left|\sum_{n=1}^{\lfloor{jt^{1/5}\rfloor}-\delta_0}\frac{1}{n^{1+it}}\right|+\left|\sum_{n=\lceil{jt^{1/5}}\rceil}^{\lfloor{\sqrt{\frac{t}{2\pi}}}\rfloor}\frac{1}{n^{1+it}}\right|+\frac{g(t)}{t^{1/2}}\left\lvert \sum_{n=1}^{\lfloor{\sqrt{\frac{t}{2\pi}}}\rfloor}\frac{1}{n^{-it}}\right\rvert+\mathcal{R}. \label{eq1} 
\end{align}
where $ j\in \mathbb{Z}_{> 1}$ to be chosen later 
and 
\[\delta_0:= \begin{cases} 
      1 & jt^{1/5}\in \mathbb{Z} \\
      0 & \text{otherwise}
   \end{cases}
\]
The first sum in $(\ref{eq1})$ in the range $1\leq n\leq \lfloor jt^{1/5}\rfloor-\delta_0$ is computed using triangle inequality and explicit bound on harmonic sum obtained using partial summation and stated in \cite{trudgian}
\begin{equation}
    \sum_{n\leq N} \frac{1}{n}\leq \log N+\gamma+\frac{1}{N}. \label{harmonicsum}
\end{equation}
The second sum in the range $\lceil {jt^{1/5}}\rceil\leq n\leq \lfloor \sqrt{t/(2\pi)}\rfloor$ is divided into dyadic pieces where each of these pieces is estimated using Lemma $\ref{fifthderivative}$. Lastly, the third sum is bounded trivially using triangle inequality.

The dyadic subdivision for this second  sum in $(\ref{eq1})$ is carried out in the following manner and then using triangle inequality subject to $\epsilon>\displaystyle\frac{2}{jt^{1/5}}>0$:
\begin{equation}
   \sum_{n=\lceil{jt^{1/5}}\rceil}^{\lfloor{\sqrt{t/2\pi}\rfloor}}\frac{1}{n^{1+it}}= \sum_{r=0}^{R(\epsilon)-1}\sum_{n=\lceil (1+\epsilon)^rjt^{1/5}\rceil}^{\lfloor (1+\epsilon)^{r+1}jt^{1/5}\rfloor-\delta_r} \frac{1}{n^{1+it}} \label{eq2}
\end{equation}
where,
\[\delta_r:= \begin{cases} 
      1 & (1+\epsilon)^{r+1}jt^{1/5}\in \mathbb{Z} \\
      0 & \text{otherwise}
   \end{cases}
\]
and
\begin{equation}
    R(\epsilon)\leq \left\lfloor\frac{\frac{3}{10}\log t-\log {\sqrt{2\pi}j}}{\log(1+\epsilon)}\right\rfloor+1  \label{nump}
\end{equation}
along with $R(\epsilon)\geq 1.$ Here, $R(\epsilon)$ gives the number of dyadic pieces of our main sum obtained using the parameter $\epsilon$. This bound is obtained using the following inequality and solving for $R(\epsilon)$: 
\begin{equation*}
  (1+\epsilon)^{R(\epsilon)-1}jt^{1/5}  \leq \lceil(1+\epsilon)^{R(\epsilon)-1}jt^{1/5}\rceil \leq \left\lfloor \sqrt{\frac{t}{2\pi}} \right\rfloor \leq \sqrt{\frac{t}{2\pi}}.
\end{equation*}
Also note that that since $\epsilon>\displaystyle\frac{2}{jt^{1/5}}$ we have $\lceil (1+\epsilon)^rjt^{1/5}\rceil< \lfloor (1+\epsilon)^{r+1}jt^{1/5}\rfloor-\delta_r$.

To bound the sum on the right hand side in ($\ref{eq2}$) we first use Lemma \ref{ps} and thus we get
\begin{equation}
S\coloneqq \left|\sum_{r=0}^{R(\epsilon)-1} \sum_{n=\lceil (1+\epsilon)^rjt^{1/5}\rceil}^{\lfloor (1+\epsilon)^{r+1}jt^{1/5}\rfloor-\delta_r} \frac{1}{n^{1+it}}\right|\leq \sum_{r=0}^{R(\epsilon)-1}\frac{1}{\lceil (1+\epsilon)^rjt^{1/5}\rceil}\max_{1\leq \Delta\leq L}\left| \sum_{n=\lceil (1+\epsilon)^rjt^{1/5}\rceil}^{\lceil (1+\epsilon)^rjt^{1/5}\rceil+\Delta-1} n^{-it} \right| \label{eq3}
\end{equation}
where $L\coloneqq L_r=\lfloor (1+\epsilon)^{r+1}jt^{1/5}\rfloor-\delta_r-\lceil (1+\epsilon)^rjt^{1/5}\rceil+1$ where $L> 1$ because $\epsilon>\displaystyle\frac{2}{jt^{1/5}}$.

We will apply Lemma $\ref{fifthderivative}$ to the inner sum in $(\ref{eq3})$. To do so first note,
\begin{equation*}
    f(x):=\frac{-t}{2\pi}\log {x}, \hspace{15mm}\lceil (1+\epsilon)^rjt^{1/5}\rceil\leq x\leq \lceil (1+\epsilon)^rjt^{1/5}\rceil+(\Delta-1),
\end{equation*}
then 
\begin{equation*}
    f^{(4)}(x):=\frac{-12t}{\pi x^5}, \hspace{15mm}\lceil (1+\epsilon)^rjt^{1/5}\rceil\leq x\leq \lceil (1+\epsilon)^rjt^{1/5}\rceil+(\Delta-1)).
\end{equation*}
Now using $1\leq \Delta\leq L$ we have,
\begin{align}
    \frac{12t}{\pi (\lceil (1+\epsilon)^rjt^{1/5}\rceil+\Delta-1)^5}&\leq |f^{(5)}(x)|\leq  \frac{12t}{\pi \lceil (1+\epsilon)^rjt^{1/5}\rceil^5}\\
     \frac{12t}{\pi (\lfloor (1+\epsilon)^{r+1}jt^{1/5}\rfloor)^4}&\leq |f^{(5)}(x)|\leq  \frac{12t}{\pi \lceil (1+\epsilon)^rjt^{1/5}\rceil^5}\\
    \frac{1}{W}&\leq |f^{(5)}(x)|\leq  \lambda\frac{1}{W}
\end{align}

Thus, we set
\begin{equation*}
    W:=W_r=\frac{\pi (\lfloor (1+\epsilon)^{r+1}jt^{1/5}\rfloor)^5}{12t}, \hspace{15mm} \lambda:=\lambda_r=\frac{(\lfloor (1+\epsilon)^{r+1}jt^{1/5}\rfloor)^5}{(\lceil (1+\epsilon)^rjt^{1/5}\rceil)^5}
\end{equation*}

Using $\lceil (1+\epsilon)^{r}jt^{1/5}\rceil\leq \lfloor (1+\epsilon)^{r+1}jt^{1/5}\rfloor-\delta_r\leq\lfloor (1+\epsilon)^{r+1}jt^{1/5}\rfloor$ and $r\geq 0, j>1$ we conclude $W>1$ and $\lambda\geq 1$. Hence we can apply Lemma $\ref{fifthderivative}$ to the right-hand side sum in $(\ref{eq3})$ and get
\begin{align*}
    S&\leq \sum_{r=0}^{R(\epsilon)-1}\frac{1}{\lceil (1+\epsilon)^rjt^{1/5}\rceil}(\alpha_5L^2W^{-1/15}+\tau_5L^{7/4}W^{1/15}+\gamma_5L^{3/2}W^{2/15}+\omega_5L^{5/4}W^{2/15}+\beta_5LW^{3/15}\\
    &\hspace{5mm}+\eta_5\alpha_5L+\eta_5\tau_5L^{3/4}W^{2/15}+\eta_5\gamma_5L^{1/2}W^{3/15}+\eta_5\omega_5L^{1/4}W^{1/5}+\eta_5\beta_5W^{4/15})^{1/2}
\end{align*}
Factoring $L^2W^{-1/15}$ gives us that  
\begin{align}
    S&\leq \sum_{r=0}^{R(\epsilon)-1}\frac{\sqrt{L^2W^{-1/15}}}{\lceil (1+\epsilon)^rjt^{1/5}\rceil} (\alpha_5+\tau_5L^{-1/4}W^{2/15}+\gamma_5L^{-1/2}W^{3/15}+\omega_5L^{-3/4}W^{3/15}+\beta_5L^{-1}W^{4/15}\nonumber\\
    &\hspace{5mm}+\eta_5\alpha_5L^{-1}W^{1/15}+\eta_5\tau_5L^{-5/4}W^{3/15}+\eta_5\gamma_5L^{-3/2}W^{4/15}+\eta_5\omega_5L^{-7/4}W^{4/15}+\eta_5\beta_5L^{-2}W^{5/15})^{1/2} \label{s1}
\end{align}

Next let us bound $L^2W^{-1/15}$ using the definition of $L$ and $W$. With this in mind and using $\lfloor x\rfloor\leq x$, $\lceil x \rceil\geq x$ and $\delta_r\geq 0$ we obtain
\begin{align}
    L&=\lfloor (1+\epsilon)^{r+1}jt^{1/5}\rfloor-\delta_r-\lceil (1+\epsilon)^rjt^{1/5}\rceil+1\\
    &\leq \epsilon(1+\epsilon)^rjt^{1/5}\psi_{\epsilon, r, t} \label{lubound}
\end{align}
where 
\begin{equation}
\psi_{\epsilon, r, t}\coloneqq 1+\frac{1}{\epsilon(1+\epsilon)^rjt^{1/5}} \label{defp}
\end{equation}

With this at hand we now bound $LW^{-1/15}$ from above using $\lfloor (1+\epsilon)^{r+1}jt^{1/5}\rfloor\geq\lfloor (1+\epsilon)^{r+1}jt^{1/5}\rfloor-\delta_r\geq \lceil (1+\epsilon)^rjt^{1/5} \rceil\geq (1+\epsilon)^rjt^{1/5}$ and get
\begin{align} 
    LW^{-1/15}&\leq \epsilon(1+\epsilon)^rjt^{1/5}\psi_{\epsilon, r, t}\left(\frac{\pi (\lfloor (1+\epsilon)^{r+1}jt^{1/5}\rfloor)^5}{12t}\right)^{-1/15}\\
     &\leq \epsilon\psi_{\epsilon, r, t}\left(\frac{12}{\pi}\right)^{1/15}j^{2/3}(1+\epsilon)^{2r/3}t^{1/5} \label{ubl}.
\end{align}

Substituting $(\ref{ubl}) $ and $(\ref{lubound})$ in $(\ref{s1})$ we have 

\begin{align}
    S&\leq \left(\frac{12}{\pi}\right)^{1/30}\epsilon\sum_{r=0}^{R(\epsilon)-1}\frac{\psi_{\epsilon, r, t}}{j^{1/6}(1+\epsilon)^{r/6}}\mathcal{B}_r \label{sumb0}
\end{align}
where 
\begin{align}
    \mathcal{B}_r&=\Bigg(\alpha_5+\tau_5\frac{W^{2/15}}{L^{1/4}}+\gamma_5\frac{W^{3/15}}{L^{1/2}}+\omega_5\frac{W^{3/15}}{L^{3/4}}+\beta_5\frac{W^{4/15}}{L}+\eta_5\alpha_5\frac{W^{1/15}}{L}+\eta_5\tau_5\frac{W^{3/15}}{L^{5/4}}\nonumber\\
    &\hspace{5mm}+\eta_5\gamma_5\frac {W^{4/15}}{L^{3/2}}+\eta_5\omega_5\frac{W^{4/15}}{L^{7/4}}+\eta_5\beta_5\frac{W^{5/15}}{L^2}\Bigg)^{1/2} \label{bb}
\end{align}

Now let us focus on bounding $\mathcal{B}_r$ using upper and lower bounds for $W$ and $L$ respectively. 

We have the following upper bound for $W$:
\begin{align}
    W&= \frac{\pi}{12}\frac{\lfloor(1+\epsilon)^{r+1}jt^{1/5}\rfloor^{5}}{t}\leq \frac{\pi}{12}(1+\epsilon)^5j^5(1+\epsilon)^{5r} \label{ubw}
\end{align}

and using $\lfloor x \rfloor\geq x-1$, $\lceil x \rceil\leq x+1$ and $\delta_r\leq 1$
\begin{align}
    L&=\lfloor (1+\epsilon)^{r+1}jt^{1/5}\rfloor-\delta_r-\lceil (1+\epsilon)^rjt^{1/5}\rceil+1 \nonumber\\
    &\geq(1+\epsilon)^{r}jt^{1/5}\left(\epsilon-\frac{2}{(1+\epsilon)^{r}jt^{1/5}}\right)\nonumber\\
    &\geq \phi_{\epsilon, j, t}(1+\epsilon)^{r}jt^{1/5} \label{lbl}
\end{align}
where $\phi_{\epsilon, j, t}=\displaystyle\left(\epsilon-\frac{2}{jt^{1/5}}\right)$ and recall that $\epsilon>\displaystyle\frac{2}{jt^{1/5}}$.

Now using $(\ref{ubw})$ and $(\ref{lbl})$ we can bound for a quantity in the form
\begin{align}
    \frac{W^{a/15}}{L^b}&\leq \left(\frac{\pi}{12}\right)^{a/15}\frac{(1+\epsilon)^{a/3}j^{a/3-b}}{\phi_{\epsilon, j, t}^{b}}(1+\epsilon)^{(a/3-b)r}\frac{1}{t^{b/5}} \label{genb}
\end{align}
Using $(\ref{genb})$ we can bound the following quantities:
\begin{equation*}
\frac{W^{2/15}}{L^{1/4}}, \frac{W^{1/5}}{L^{1/2}}, \frac{W^{1/5}}{L^{3/4}}, \frac{W^{4/15}}{L}, \frac{W^{1/15}}{L}, \frac{W^{1/5}}{L^{5/4}}, \frac {W^{4/15}}{L^{3/2}}, \frac{W^{4/15}}{L^{7/4}}, \frac{W^{1/3}}{L^2}
\end{equation*}
by substituting $a= 1, 2, 3, 4, 5$ and $b=\frac{1}{4}, \frac{1}{2}, \frac{3}{4}, 1, \frac{5}{4}, \frac{3}{2}, \frac{7}{4}, 2$.

With $(\ref{bb})$, $(\ref{genb})$, the above mentioned $a, b$ values and the inequality $\sqrt{x_1+x_2+\ldots}\leq \sqrt{x_1}+\sqrt{x_2}+\ldots$ for $x_1, x_2, \ldots>0$, we finally obtain the following complicated upper bound for $\mathcal{B}_r$
\begin{align}
    \mathcal{B}_r&\leq \sqrt{\alpha_5}+\left(\tau_5\left(\frac{\pi}{12}\right)^{2/15}\frac{(1+\epsilon)^{2/3}j^{5/12}}{\phi_{\epsilon, j, t}^{1/4}}\frac{(1+\epsilon)^{5r/12}}{t^{1/20}}+\gamma_5\left(\frac{\pi}{12}\right)^{1/5}\frac{(1+\epsilon)j^{1/2}}{\phi_{\epsilon, j, t}^{1/2}}\frac
{(1+\epsilon)^{r/2}}{t^{1/10}}\right. \nonumber\\
&\qquad \left. +\omega_5\left(\frac{\pi}{12}\right)^{1/5}\frac{(1+\epsilon)j^{1/4}}{\phi_{\epsilon, j, t}^{3/4}}\frac{(1+\epsilon)^{r/4}}{t^{3/20}}+\beta_5\left(\frac{\pi}{12}\right)^{4/15}\frac{(1+\epsilon)^{4/3}j^{1/3}}{\phi_{\epsilon, j, t}}\frac{(1+\epsilon)^{r/3}}{t^{1/5}} \right)^{1/2}\nonumber\\
&\qquad +\left(\eta_5\alpha_5\left(\frac{\pi}{12}\right)^{1/15}\frac{(1+\epsilon)^{1/3}}{j^{2/3}\phi_{\epsilon, j, t}}\frac{1}{(1+\epsilon)^{2r/3}t^{1/5}}+\eta_5\tau_5\left(\frac{\pi}{12}\right)^{1/5}\frac{(1+\epsilon)}{j^{1/4}\phi_{\epsilon, j, t}^{5/4}}\frac{1}{(1+\epsilon)^{r/4}t^{1/4}}\right.\nonumber\\
&\qquad \left.+\eta_5\gamma_5\left(\frac{\pi}{12}\right)^{4/15}\frac{(1+\epsilon)^{4/3}}{j^{1/6}\phi_{\epsilon, j, t}^{3/2}}\frac{1}{(1+\epsilon)^{r/6}t^{3/10}} + \eta_5\omega_5\left(\frac{\pi}{12}\right)^{4/15}\frac{(1+\epsilon)^{4/3}}{j^{5/12}\phi_{\epsilon, j, t}^{7/4}}\frac{1}{(1+\epsilon)^{5r/12}t^{7/20}}\right. \nonumber\\
&\qquad \left. \eta_5\beta_5\left(\frac{\pi}{12}\right)^{1/3}\frac{(1+\epsilon)^{5/3}}{\phi_{\epsilon, j, t}^2j^{1/3}}\frac{1}{(1+\epsilon)^{r/3}t^{2/5}}\right)^{1/2} \label{boundonb}
\end{align}

At this stage we can factor $\displaystyle\frac{(1+\epsilon)^{r/2}}{t^{1/20}}$ and $\displaystyle\frac{1}{(1+\epsilon)^{r/6}t^{1/5}}$ from the second and third square-root terms above and plugging the bound obtained for $\mathcal{B}_r$ into inequality $(\ref{sumb0})$ we obtain

\begin{align}
   S &\leq \left(\frac{12}{\pi}\right)^{1/30}\frac{\epsilon}{j^{1/6}}\sum_{r=0}^{R(\epsilon)-1}\psi_{\epsilon, r, t}\Bigg(\frac{\sqrt{\alpha_5}}{(1+\epsilon)^{r/6}}+\mathcal{C}_1\frac{(1+\epsilon)^{r/12}}{t^{1/40}}+\mathcal{C}_2\frac{1}{(1+\epsilon)^{r/4}t^{1/10}}\Bigg) \label{blast}
\end{align}
where 
\begin{align*}
    \mathcal{C}_1\coloneqq\mathcal{C}_1(t)&= \Bigg(\tau_5\left(\frac{\pi}{12}\right)^{2/15}\frac{(1+\epsilon)^{2/3}j^{5/12}}{\phi_{\epsilon, j, t}^{1/4}}+\gamma_5\left(\frac{\pi}{12}\right)^{1/5}\frac{(1+\epsilon)j^{1/2}}{\phi_{\epsilon, j, t}^{1/2}}\frac{1}{t^{1/20}}+\omega_5\left(\frac{\pi}{12}\right)^{1/5}\frac{(1+\epsilon)j^{1/4}}{\phi_{\epsilon, j, t}^{3/4}}\frac{1}{t^{1/10}}\\
    &\hspace{8mm}+\beta_5\left(\frac{\pi}{12}\right)^{4/15}\frac{(1+\epsilon)^{4/3}j^{1/3}}{\phi_{\epsilon, j,t}}\frac{1}{t^{3/20}}\Bigg)^{1/2}\\
    \mathcal{C}_2\coloneqq\mathcal{C}_2(t)&= \Bigg(\eta_5\alpha_5\left(\frac{\pi}{12}\right)^{1/15}\frac{(1+\epsilon)^{1/3}}{j^{2/3}\phi_{\epsilon, j,t}}+\eta_5\tau_5\left(\frac{\pi}{12}\right)^{1/5}\frac{(1+\epsilon)}{j^{1/4}\phi_{\epsilon, j,t}^{5/4}}\frac{1}{t^{1/20}}+\eta_5\gamma_5\left(\frac{\pi}{12}\right)^{4/15}\frac{(1+\epsilon)^{4/3}}{j^{1/6}\phi_{\epsilon, j,t}^{3/2}}\frac{1}{t^{1/10}}\\
    &\hspace{8mm}+\eta_5\omega_5\left(\frac{\pi}{12}\right)^{4/15}\frac{(1+\epsilon)^{4/3}}{j^{5/12}\phi_{\epsilon, j,t}^{7/4}}\frac{1}{t^{3/20}}+\eta_5\beta_5\left(\frac{\pi}{12}\right)^{1/3}\frac{(1+\epsilon)^{5/3}}{\phi_{\epsilon, j,t}^2j^{1/3}} \frac{1}{t^{1/5}}\Bigg)^{1/2}
\end{align*}

Next using the definition of $\psi_{\epsilon, r, t}$ from $(\ref{defp})$ we can write $(\ref{blast})$ as
\begin{align}
S&\leq\left(\frac{12}{\pi}\right)^{1/30}\frac{\epsilon}{j^{1/6}}\sum_{r=0}^{R(\epsilon)-1}\Bigg(\frac{\sqrt{\alpha_5}}{(1+\epsilon)^{r/6}}+\frac{\mathcal{C}_1(1+\epsilon)^{r/12}}{t^{1/40}}+\frac{\mathcal{C}_2}{(1+\epsilon)^{r/4}t^{1/10}}\Bigg)\nonumber\\
&\hspace{5mm}+\left(\frac{12}{\pi}\right)^{1/30}\frac{1}{j^{7/6}}\sum_{r=0}^{R(\epsilon)-1}\Bigg(\frac{\sqrt{\alpha_5}}{(1+\epsilon)^{7r/6}t^{1/5}}+\frac{\mathcal{C}_1}{(1+\epsilon)^{11r/12}t^{9/40}}+\frac{\mathcal{C}_2}{(1+\epsilon)^{5r/4}t^{3/10}}\Bigg) \label{almostf}
\end{align}

Let us estimate each of the above sums using the following inequality that is valid for $c, d>0$ and using $(\ref{nump}):$
\begin{align}
    \sum_{r=0}^{R(\epsilon)-1}\frac{1}{(1+\epsilon)^{cr}}&\leq \frac{(1+\epsilon)^{c}}{(1+\epsilon)^{c}-1}, \hspace{2mm} \sum_{r=0}^{R(\epsilon)-1}(1+\epsilon)^{dr}\leq \frac{(1+\epsilon)^{d}}{(2\pi)^{d/2}j^{d}((1+\epsilon)^{d}-1)}t^{3d/10}\label{sumf}
\end{align}
Finally we can obtain a bound for $S$ using $(\ref{almostf})$ and $(\ref{sumf}):$
\begin{align*}
    S&\leq d_1\Bigg(\sqrt{\alpha_5}\frac{(1+\epsilon)^{1/6}}{(1+\epsilon)^{1/6}-1}+\frac{\mathcal{C}_1(1+\epsilon)^{1/12}}{(2\pi)^{1/24}j^{1/12}((1+\epsilon)^{1/12}-1)}+\frac{\mathcal{C}_2(1+\epsilon)^{1/4}}{(1+\epsilon)^{1/4}-1}t^{-1/10}\Bigg)\\
    &+d_2\Bigg(\frac{\sqrt{\alpha_5}(1+\epsilon)^{7/6}}{(1+\epsilon)^{7/6}-1}t^{-1/5}+\frac{\mathcal{C}_1(1+\epsilon)^{11/12}}{(1+\epsilon)^{11/12}-1}t^{-9/40}+\frac{\mathcal{C}_2(1+\epsilon)^{5/4}}{(1+\epsilon)^{5/4}-1}t^{-3/10}\Bigg).
\end{align*}

where
\begin{align*}
    d_1\coloneqq \left(\frac{12}{\pi}\right)^{1/30}\frac{\epsilon}{j^{1/6}}, \hspace{8mm} d_2\coloneqq \left(\frac{12}{\pi}\right)^{1/30}\frac{1}{j^{7/6}}.
\end{align*}
In the end we have 
\begin{align}
    \left|\sum_{n=\lceil{jt^{1/5}}\rceil}^{\lfloor{\sqrt{\frac{t}{2\pi}}}\rfloor}\frac{1}{n^{1+it}}\right|\leq A_1t^{-3/10}+A_2t^{-9/40}+A_3t^{-1/5}+A_4t^{-1/10}+A_5 \label{f1}
\end{align}
where 
\begin{align*}
    A_1&\coloneqq\frac{d_2\mathcal{C}_2(1+\epsilon)^{5/4}}{(1+\epsilon)^{5/4}-1}, \hspace{3mm} A_2\coloneqq\frac{d_2\mathcal{C}_1(1+\epsilon)^{11/12}}{(1+\epsilon)^{11/12}-1}, \hspace{3mm} A_3\coloneqq \frac{d_2\sqrt{\alpha_5}(1+\epsilon)^{7/6}}{(1+\epsilon)^{7/6}-1}, \hspace{3mm} A_4\coloneqq \frac{d_1\mathcal{C}_2(1+\epsilon)^{1/4}}{(1+\epsilon)^{1/4}-1},\\
     A_5&\coloneqq \frac{d_1\sqrt{\alpha_5}(1+\epsilon)^{1/6}}{(1+\epsilon)^{1/6}-1}+\frac{d_1\mathcal{C}_1(1+\epsilon)^{1/12}}{(2\pi)^{1/24}j^{1/12}((1+\epsilon)^{1/12}-1)}.
\end{align*}

At this stage we bound the third sum in (\ref{eq1}) trivially and thus we get the following bound valid for $t>(2/j)^5$
\begin{align}
     \lvert \zeta(1+it)\rvert \leq \gamma+\log j+\frac{1}{5}\log t+\frac{1}{jt^{1/5}-2} &+A_1t^{-3/10}+A_2t^{-9/40}+A_3t^{-1/5}\nonumber\\
     &+A_4t^{-1/10}+A_5+\frac{g(t)}{\sqrt{2\pi}}+\mathcal{R} \label{secondsumtrivial}
\end{align}
On the other hand, we also have the option of bounding the last sum in $(\ref{eq1})$ using Lemma $\ref{fifthderivative}$ and obtaining an estimate for it using the fact that $\displaystyle\left\lvert \sum_{n=1}^{N}\frac{1}{n^{-it}}\right\rvert=\left\lvert \sum_{n=1}^{N}\frac{1}{n^{it}}\right\rvert$ and following a similar interval split and then steps taken to estimate $S$ as in $(\ref{s1})$. However, the improvements obtained in this case is negligible unless $t$ is astronomically large. Hence we omit the details of such a computation here.  

Finally, we choose the following values for our parameters via numerical experimentation:
\begin{align*}
    \epsilon=0.32, \hspace{2mm} j=60,\hspace{2mm} \eta_3= \left(\frac{15\sqrt{\pi}}{32}\right)^{2/3}, \hspace{2mm}\eta_4=\left(\frac{91}{72\sqrt{23}}\right)^{6/7}, \hspace{2mm}\eta_5=2.2, \hspace{2mm}t_0=8\times10^{60}
\end{align*}
Additionally, we use that
 \begin{equation}
     W\geq \frac{\pi j^5}{12}, \hspace{2mm} \lambda\leq (1+\epsilon)^5, \hspace{2mm} \mathcal{C}_1\leq \mathcal{C}_1(t_0), \hspace{2mm} \mathcal{C}_2\leq \mathcal{C}_2(t_0), \hspace{2mm} g(t)\leq g(t_0), \hspace{2mm} \mathcal{R}\leq\mathcal{R}(t_0)
 \end{equation}
and plugging them into $(\ref{secondsumtrivial})$ we get 
\begin{equation}
    |\zeta(1+it)|\leq 43.9259 + \frac{1}{60t^{1/5}-2}+\frac{0.035264}{t^{3/10}}+ \frac{0.255693}{t^{9/40}}+\frac{0.0552644}{t^{1/5}}+ \frac{2.96078}{t^{1/10}}+ \frac{1}{5}\log t \label{constzeta}
\end{equation}
Now from $(\ref{constzeta})$ we deduce that 
\begin{equation}
    |\zeta(1+it)|-\frac{1}{5}\log t \leq F(t).    
\end{equation}

where 
\begin{equation*}
    F(t)\coloneqq 43.9259 + \frac{1}{60t^{1/5}-2}+\frac{0.035264}{t^{3/10}}+ \frac{0.255693}{t^{9/40}}+\frac{0.0552644}{t^{1/5}}+ \frac{2.96078}{t^{1/10}}
\end{equation*}
Note that $F(t)$ is decreasing in $t$ and thus for $t\geq t_0$ and  we have that 
\begin{equation*}
    F(t)\leq F(t_0)\leq 44.02
\end{equation*}
giving us 
\begin{equation}
    |\zeta(1+it)|\leq \frac{1}{5}\log t+44.02 \label{mine}
\end{equation}
Now for $t\geq 47$, we use $(\ref{zetatriangle})$ and get that 
\begin{equation}
    |\zeta(1+it)|\leq \frac{1}{2}\log t+1.93. \label{consttriv}
\end{equation}
Lastly for $t\geq 3$ we also have
\begin{equation}
    |\zeta(1+it)|\leq \log t. \label{back}
\end{equation}

Ultimately combining $(\ref{mine}), (\ref{consttriv})$ and $(\ref{back})$ for $t\geq 3$ we obtain the the following bound for the zeta function on $1$-line:
\begin{equation}
    |\zeta(1+it)|\leq \min\left(\log t, \frac{1}{2}\log t+1.93, \frac{1}{5}\log t+44.02\right)
\end{equation}

\section{Concluding remarks}
Note that choosing a different set of parameters, one can obtain a slight improvement on the constant $44.02$ in $(\ref{actualzeta1const})$. However it seems that this constant cannot be improved beyond the Euler constant $\gamma=0.57721\ldots$ as the harmonic sums $(\ref{harmonicsum})$ are bounded above and below by at least $\log N+\gamma$. Moreover the leading constant $\frac{1}{5}$ in $(\ref{eventualconst})$ is the best that can be obtained if one insists on using Lemma $\ref{fifthderivative}$ (which is a fifth derivative test) and inequality $(\ref{harmonicsum})$ for harmonic sum, as done here. However, this $1/5$ could be improved if one uses higher explicit derivative test or if one finds a suitable method to take advantage of possible cancellations while estimating  the first sum on the r.h.s in $(\ref{eq1})$ instead of using triangle inequality followed by  $(\ref{harmonicsum})$. Estimates on $|\zeta(1+it)|$ using higher explicit derivative tests are in preparation by the author. 
\bibliographystyle{amsplain}
\bibliography{mainbound}
\end{document}